\documentclass[a4paper,10pt]{amsart}

\usepackage[english]{babel}
\usepackage[utf8]{inputenc}
\usepackage{hyperref}
\usepackage{amssymb}
\usepackage[enableskew]{youngtab}
\usepackage{young}

\newcommand{\bs}{\blacksquare}

\theoremstyle{plain}
\newtheorem{question}{Question}
\newtheorem{conjecture}[question]{Conjecture}
\newtheorem{theorem}[question]{Theorem}
\newtheorem{corollary}[question]{Corollary}
\newtheorem{proposition}[question]{Proposition}
\newtheorem{lemma}[question]{Lemma}
\newtheorem*{mainquestion}{Main Question}
\newtheorem*{stanleyconjecture}{Stanley's Conjecture}

\theoremstyle{remark}
\newtheorem{remark}[question]{Remark}

\begin{document}

\title{Some combinatorial properties of skew Jack symmetric functions}
\author{Paolo Bravi, Jacopo Gandini}


\maketitle

\begin{abstract}
Motivated by Stanley's conjecture on the multiplication of Jack symmetric functions, we prove a couple of identities showing that skew Jack symmetric functions are semi-invariant up to translation and rotation of a $\pi$ angle of the skew diagram. It follows that, in some special cases, the coefficients of the skew Jack symmetric functions with respect to the basis of the monomial symmetric functions are polynomials with nonnegative integer coefficients.  
\end{abstract}

\section{Introduction}

Jack symmetric functions $J_\lambda(x;\alpha)$ form a basis of the ring of symmetric functions in the infinite (countable) set of indeterminates $x=(x_1,x_2,\ldots)$ with coefficients in the fraction field $\mathbb Q(\alpha)$, as $\lambda=(\lambda_1,\lambda_2,\ldots)$ varies in the set of (integer) partitions. We recall here some notation and basic facts from \cite{M,S,KS}.

The symmetric functions $J_\lambda=J_\lambda(x;\alpha)$ are uniquely defined by the following properties:
\begin{enumerate}
\item they are pairwise orthogonal, that is, $\langle J_\lambda,J_\mu\rangle=0$ for all $\lambda\neq\mu$;
\item triangular with respect to the monomial symmetric functions $m_\lambda(x)$, that is, 
\[J_\lambda(x;\alpha)=\sum_{\mu\leq\lambda}v_{\lambda,\mu}(\alpha)\,m_\mu(x);\]
\item normalized as $v_{\lambda,(1^n)}(\alpha)=n!$ if $|\lambda|=n$.
\end{enumerate}
The partial ordering among partitions is the usual dominance order. The scalar product is uniquely defined by the following properties: the power sum symmetric functions $p_\lambda(x)$ are pairwise orthogonal and
\[\langle p_\lambda(x),p_\lambda(x)\rangle=z_\lambda \alpha^{\ell(\lambda)},\]
where $z_\lambda=(1^{k_1}2^{k_2}\cdots)(k_1!k_2!\cdots)$ if $k_i=k_i(\lambda)$ denotes the number of parts of $\lambda$ equal to $i$, and $\ell(\lambda)$ is the length of $\lambda$ (the number of nonzero parts of $\lambda$). Notice that specializing $\alpha=1$ one has the usual scalar product on symmetric functions, therefore the Jack symmetric functions specialize to scalar multiples of the Schur symmetric functions $s_\lambda(x)$.  

As conjectured by I.G.\ Macdonald  (\cite[VI (10.26?)]{M}, \cite[Conjecture 8.1]{S}) and proved by F.\ Knop and S.\ Sahi \cite{KS}, the above functions $v_{\lambda,\mu}(\alpha)$ are polynomials in $\alpha$ with nonnegative integer coefficients. Furthermore, Knop and Sahi have found an explicit integral combinatorial formula for $v_{\lambda,\mu}(\alpha)$ in terms of certain {\it admissible} fillings of weight $\mu$ of the Young diagram of $\lambda$.

As for skew Schur symmetric functions, skew Jack symmetric functions $J_{\lambda/\mu}=J_{\lambda/\mu}(x;\alpha)$ are defined by the following identities, for all partitions $\nu$:
\[\langle J_{\lambda/\mu}, J_\nu\rangle = \langle J_\lambda, J_\mu J_\nu\rangle.\]
Therefore,
\[J_{\lambda/\mu}=\sum_\nu\frac{\langle J_\lambda,J_\mu J_\nu\rangle}{\langle J_\nu,J_\nu\rangle}J_\nu,\]
the sum is clearly finite as the coefficient of $J_\nu$ can be nonzero only if $|\nu|=|\lambda|-|\mu|$. 

For a partition $\lambda$, set $u_\lambda = \prod_{i \geq 1} k_i!$, where $k_i$ denotes the number of parts of $\lambda$ which are equal to $i$ as above.

Let
\[J_{\lambda/\mu}(x;\alpha)=\sum_\nu v_{\lambda/\mu,\, \nu}(\alpha)\,m_\nu(x)\]
and let $\tilde v_{\lambda/\mu,\, \nu}(\alpha) = u_\nu^{-1} \ v_{\lambda/\mu,\, \nu}(\alpha)$. As far as we know the following was never considered.
 
\begin{mainquestion}	
Are the $\tilde v_{\lambda/\mu,\, \nu}(\alpha)$ polynomials with nonnegative integer coefficients?
\end{mainquestion}

By explicit computation with SageMath \cite{SAGE} we know the answer is affirmative for $|\lambda|\leq10$.

A complete affirmative answer to the above question and especially an explicit integral combinatorial interpretation of the functions $v_{\lambda/\mu,\,\nu}(\alpha)$ could be seen as a first step toward a possible proof of the following interesting and still open conjecture made by R.\ Stanley.

\begin{stanleyconjecture}{\cite[Conjecture 8.3]{S}}
The functions $\langle J_\lambda, J_\mu J_\nu\rangle$ are polynomials with nonnegative integer coefficients.
\end{stanleyconjecture}
 
It is already known that the functions $g^\lambda_{\mu, \nu}(\alpha) := \langle J_\lambda, J_\mu J_\nu\rangle$ are polynomials in $\alpha$ with integer coefficients, this follows from the fact that the $v_{\lambda,\mu}(\alpha)$ are polynomials with integer coefficients. We will refer to the $g^\lambda_{\mu, \nu}(\alpha)$ as the {\emph{Stanley $g$-polynomials}.

Let us briefly comment on Stanley's conjecture. For any couple of partitions $\mu$ and $\nu$ we clearly have
\[J_{\mu}\,J_\nu=\sum_\lambda \frac{\langle J_\lambda,J_\mu J_\nu\rangle}{\langle J_\lambda,J_\lambda\rangle}J_\lambda.\]
The squared norm of $J_\lambda$, denoted by $j_\lambda(\alpha)$, is known to be a polynomial with nonnegative integer coefficients, see Theorem \ref{norma} below.  Therefore, from Stanley's conjecture we would have that the nonvanishing of the Littlewood-Richardson coefficient $c^\lambda_{\mu,\nu}$ in
\[s_{\mu}\,s_\nu=\sum_\lambda c^\lambda_{\mu,\nu}s_\lambda\]
implies the nonvanishing of the coefficient $g^\lambda_{\mu,\nu}/j_\lambda$ in 
\[J_{\mu}\,J_\nu=\sum_\lambda \frac{g^\lambda_{\mu,\nu}}{j_\lambda}J_\lambda\]
for all $\alpha>0$. This would provide a quite complete information on the multiplication of spherical functions of certain symmetric spaces. Indeed, for certain values of $\alpha>0$, such as $\alpha=1/2$ or $\alpha=2$, Jack symmetric functions specialize to certain (restricted) spherical functions, such as the so-called zonal polynomials. For further details in this direction see \cite{GH,BG}.

Our main results, which are contained in Section \ref{sec:main}, are the formulas of Theorem \ref{translation} and Theorem \ref{rotation} which allow to affirmatively answer in some special cases to the main question above, see Corollary \ref{cor:translation} and Corollary \ref{cor:rotation}. Proposition \ref{prefix} also provides affirmative answer to the main question in another special case.

Notice that Theorem \ref{translation} and Theorem \ref{rotation} are derived from a combinatorial formula due to Stanley \cite[Theorem 6.3]{S} (see Theorem \ref{SThm6.3}) which readily generalizes to the two parameter Macdonald polynomials (see \cite[VI (7.13) and (8.3)]{M}). Therefore, our formulas generalize to the Macdonald polynomials, too.

In Section \ref{sec:factorizability} we provide some remarks directly on the Stanley $g$-polynomials which under some special hypotheses are conjectured to be product of linear factors.

In Section \ref{sec:lowest} we formulate a combinatorial conjecture on the lowest coefficient of the skew Jack symmetric functions which could be the first step toward a generalization of Knop and Sahi's combinatorial formula to the skew case.

\section{Remarks and partial answers to the main question}\label{sec:main}

\subsection{The squared norm}

Let us fix some more notation. If $\lambda$ is a partition, by $s\in\lambda$ we mean that $s$ is a box in the diagram of $\lambda$. For all $s\in\lambda$ we set
\[c_{\lambda,\,s}=c_{\lambda,\,s}(\alpha)=a_{\lambda,\,s} \, \alpha+\ell_{\lambda,\,s}+1,\quad c'_{\lambda,\,s}=c'_{\lambda,\,s}(\alpha)=(a_{\lambda,\,s}+1)\alpha+\ell_{\lambda,\,s},\]
where $a_{\lambda,\,s}$ is the arm of $s$ in $\lambda$ (the number of boxes on the same row of $s$, on the right of $s$) and $\ell_{\lambda,\,s}$ is the leg of $s$ in $\lambda$ (the number of boxes on the same column of $s$, below $s$).
We will use also the following notation
\[c_{\lambda}=c_\lambda(\alpha)=\prod_{s\in\lambda}c_{\lambda,\,s}(\alpha),\quad c'_{\lambda}=c'_\lambda(\alpha)=\prod_{s\in\lambda}c'_{\lambda,\,s}(\alpha).\]

Notice that $c'_\lambda(\alpha)=\alpha^{|\lambda|}\, c_{\lambda'}(\alpha^{-1})$, where $\lambda'$ is the transposition of $\lambda$.

Let us denote by $j_\lambda=j_\lambda(\alpha)$ the squared norm of $J_\lambda(x;\alpha)$ and notice that $j_\lambda(\alpha)=v_{\lambda/\lambda,\,\emptyset}(\alpha)$.

\begin{theorem}{\cite[Theorem 5.8]{S}} 	\label{norma}
For all partitions $\lambda$, 
\[j_\lambda = c_{\lambda} \, c'_{\lambda}.\]
\end{theorem}

\subsection{Stanley's combinatorial formula}

A (rational) combinatorial formula for the function $v_{\lambda/\mu,\,\nu}(\alpha)$ has already been found by Stanley. 

Recall that a skew partition $\lambda/\mu$ is called a \textit{horizontal strip} if the corresponding skew diagram contains at most one box in every column.

A tableau $T$ of shape $\lambda/\mu$ is called {\textit{standard} if it is nondecreasing along rows and strictly increasing along columns. This is equivalent to require that by deleting the boxes labelled with $j>i$ one gets the diagram of a skew partition $\lambda^{(i)}/\mu$ such that, for all $i>1$, the skew partition $\lambda^{(i)}/\lambda^{(i-1)}$ is a horizontal strip. 

For all standard tableaux $T$ of shape $\lambda/\mu$ set 
\[w_T=w_T(\alpha)=\frac{j_\mu\prod_i\prod_{s\in\lambda^{(i)}}B_{\lambda^{(i)}/\lambda^{(i-1)},\,s}}{\prod_i\prod_{s\in\lambda^{(i-1)}}C_{\lambda^{(i)}/\lambda^{(i-1)},\,s}}\]
where
\[B_{\lambda^{(i)}/\lambda^{(i-1)},\,s}=\left\{\begin{array}{ll}
c_{\lambda^{(i)},\,s} & \text{if $\lambda^{(i)}/\lambda^{(i-1)}$ has a box in the same column as $s$}\\
c'_{\lambda^{(i)},\,s} & \text{otherwise}
\end{array}\right.\]
\[C_{\lambda^{(i)}/\lambda^{(i-1)},\,s}=\left\{\begin{array}{ll}
c_{\lambda^{(i-1)},\,s} & \text{if $\lambda^{(i)}/\lambda^{(i-1)}$ has a box in the same column as $s$}\\
c'_{\lambda^{(i-1)},\,s} & \text{otherwise}
\end{array}\right.\]

Notice that the above product is finite since
\[\frac{\prod_{s\in\lambda^{(i)}}B_{\lambda^{(i)}/\lambda^{(i-1)},\,s}}{\prod_{s\in\lambda^{(i-1)}}C_{\lambda^{(i)}/\lambda^{(i-1)},\,s}}\]
is definitely equal to $1$ for $i>\!\!>0$.

When $\lambda/\mu$ is a horizontal strip, we will also use the compact notations
$$
	B_{\lambda/\mu} = \prod_{s\in\lambda} B_{\lambda/\mu,\,s}, \qquad 	C_{\lambda/\mu} = \prod_{s\in\mu} C_{\lambda/\mu,\,s}.
$$
In particular, if $\lambda/\mu$ is a horizontal strip and $T$ is the skew tableau of shape $\lambda/\mu$ obtained by filling all the boxes with the same label, then
$$
	w_T = \frac{j_\mu B_{\lambda/\mu}}{C_{\lambda/\mu}}.
$$

Denoting as usual by $x^T$ the product of the $x_i^{k_i}$'s where $k_i$ is the number of labels of $T$ equal to $i$, one has

\begin{theorem}{\cite[Theorem 6.3]{S}}\label{SThm6.3}
For all skew partitions $\lambda/\mu$,
\[J_{\lambda/\mu}(x;\alpha)=\sum_{\substack{T\text{ standard}\\ \text{of shape }\lambda/\mu}} w_T(\alpha) x^T.\]
\end{theorem}

\subsection{The leading coefficient}

Let us first recall the formula for the leading coefficient in 
\[J_{\lambda}(x;\alpha)=\sum_\mu v_{\lambda,\, \mu}(\alpha)\,m_\mu(x).\]

\begin{theorem}{\cite[Theorem 5.6]{S}}
For all partitions $\lambda$, 
\[v_{\lambda,\lambda} = c_{\lambda}.\]
\end{theorem}

The following is a generalization of the above formula, consequence of Theorem~\ref{SThm6.3} and actually a reformulation of Proposition 8.6 in \cite{S}.

Consider the standard skew tableau $T$ of shape $\lambda/\mu$ obtained by labelling the boxes along every column with consecutive integers, that is, the boxes in column $j$ are labelled with the integers $1,\ldots,c_j$. Set also $c_j = 0$ if no label appears on column $j$. Similarly, let $r_i$ be the rightmost label in row $i$, and set $r_i = 0$ if no label appears on the row $i$.

Let $\nu_0$ be the weight of the above defined tableau $T$, then $\nu \leq \nu_0$ for all partitions $\nu$ with $v_{\lambda/\mu,\, \nu}\neq 0$, that is $v_{\lambda/\mu,\, \nu_0}$ is the leading coefficient in 
\[J_{\lambda/\mu}(x;\alpha)=\sum_\nu v_{\lambda/\mu,\, \nu}(\alpha)\,m_\nu(x).\]

\begin{proposition}{\cite[Proposition 8.6]{S}}	\label{leading-coefficient}
The leading coefficient $v_{\lambda/\mu,\, \nu_0}$ of $J_{\lambda/\mu}$ is equal to
\[
\prod_{\tiny\begin{array}{c}(i,j)\in\lambda\\ r_i \leq c_j\end{array}} c_{\lambda,(i,j)}
\prod_{\tiny\begin{array}{c}(i,j)\in\lambda\\ r_i > c_j\end{array}} c'_{\lambda,(i,j)}
\prod_{\tiny\begin{array}{c}(i,j)\in\mu\\ r_{i+c_j} > c_j\end{array}} c_{\mu,(i,j)}
\prod_{\tiny\begin{array}{c}(i,j)\in\mu\\ r_{i+c_j}\leq c_j\end{array}} c'_{\mu,(i,j)}.
\]
\end{proposition}

\begin{remark}
It follows that $\tilde v_{\lambda/\mu,\, \nu_0}$ is a polynomial with nonnegative integer coefficients. If indeed $(i,j)$ is the last element occurring in a row of $\lambda/\mu$, then by definition it holds $r_i \leq c_j$. Therefore, in the first factor appearing in the formula of Proposition \ref{leading-coefficient}, the entry $(i,j)$ gives a contribution $c_{\lambda,(i,j)} = \ell_{\lambda,(i,j)}$. On the other hand, if $(\nu_0)_k = (\nu_0)_{k+1}$, then every box in $T$ labeled by $k$ lies above a box labeled by $k+1$: thus we see that $\tilde v_{\lambda/\mu,\, \nu_0}$ is also a polynomial with nonnegative integer coefficients.
\end{remark}

\subsection{Translation of the skew diagram}

Another consequence of Theorem~\ref{SThm6.3} is the following. 

Let $\lambda/\mu$ and $\tilde \lambda/\tilde \mu$ be skew partitions, we say that the corresponding diagrams \textit{coincide up to translation} if the diagram of $\tilde \lambda/\tilde \mu$ can be obtained from the diagram of $\lambda/\mu$ by adding and/or removing \emph{empty} rows $\mu_i=\lambda_i$ to/from the top of the diagram and/or \emph{empty} columns $\mu'_j=\lambda'_j$ to/from the leftmost part of the diagram.

Fix a skew partition $\lambda/\mu$. Notice that there exists a unique minimal skew partition $\tilde \lambda/\tilde \mu$ whose diagram coincides with that $\lambda/\mu$ up to translation (see Figure~\ref{translation_example} for an example). Formally, let $i_0\geq0$ be the greatest integer such that $\mu_i=\lambda_i$ for all $1\leq i\leq i_0$ and let $j_0\geq0$ be the greatest integer such that $\mu'_j=\lambda'_j$ for all $1\leq j\leq j_0$, then for all $k>0$
\[\tilde\lambda_k=\left\{\begin{array}{ll}\lambda_{i_0+k}-j_0 & \text{if }k\leq \lambda'_{j_0+1} \\ 0 & \text{if }k>\lambda'_{j_0+1}\end{array}\right.
\text{ and }
\tilde\mu_k=\left\{\begin{array}{ll}\mu_{i_0+k}-j_0 & \text{if }k\leq \mu'_{j_0+1} \\ 0 & \text{if }k>\mu'_{j_0+1}\end{array}\right..\] 

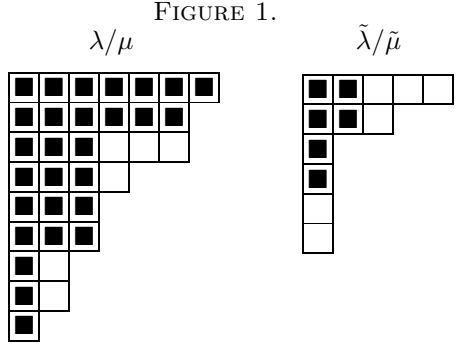
\begin{figure}\caption{}\label{translation_example}
\begin{picture}(160,120)
\put(0,0){$\young(\bs\bs\bs\bs\bs\bs\bs,\bs\bs\bs\bs\bs\bs,\bs\bs\bs\ \ \ ,\bs\bs\bs\ ,\bs\bs\bs,\bs\bs\bs,\bs\ ,\bs\ ,\bs)$} 
\put(30,110){$\lambda/\mu$} 
\put(110,33){$\young(\bs\bs\ \ \ ,\bs\bs\ ,\bs,\bs,\ ,\ )$}
\put(130,110){$\tilde\lambda/\tilde\mu$}
\end{picture}
\end{figure}

For all $s\in\mu$ set
\[c_{\mu,\lambda,\,s}=c_{\mu,\lambda,\,s}(\alpha)=a_{\mu,\,s}\alpha+\ell_{\lambda,\,s}+1,\quad c'_{\lambda,\mu,\,s}= c'_{\lambda,\mu,\,s}(\alpha)=(a_{\lambda,\,s}+1)\alpha+\ell_{\mu,\,s},\]
and set also
\[c_{\mu,\lambda}=c_{\mu,\lambda}(\alpha)=\prod_{s\in\mu}c_{\mu,\lambda,\,s},\quad c'_{\lambda,\mu}=c'_{\lambda,\mu}(\alpha)=\prod_{s\in\mu}c'_{\lambda,\mu,\,s}.\]
Notice that $c'_{\lambda,\mu}(\alpha) = \alpha^{|\mu|} \, c_{\mu',\lambda'}(\alpha^{-1})$.

\begin{theorem}\label{translation}
If the diagrams of $\lambda/\mu$ and $\tilde\lambda/\tilde\mu$ coincide up to translation, then  
\[c_{\mu,\lambda} \, c'_{\lambda,\mu} \, J_{\tilde\lambda/\tilde\mu}
=c_{\tilde\mu,\tilde\lambda} \, c'_{\tilde\lambda,\tilde\mu}\, J_{\lambda/\mu}\]
\end{theorem}

As a consequence of the previous formula, by \cite{KS} we get

\begin{corollary}\label{cor:translation}
Let $\lambda/\mu$ be a skew partition, and suppose that the diagram $\lambda/\mu$ coincides up to translation with that of a partition $\tilde\lambda=\tilde\lambda/\emptyset$. Then
\[
J_{\lambda/\mu} = c_{\mu,\lambda} \, c'_{\lambda,\mu}\, J_{\tilde \lambda}\]
In particular, $\tilde v_{\lambda/\mu,\,\nu}(\alpha)$ is a polynomial with nonnegative integer coefficients, for all partitions $\nu$. Furthermore, we have 
\[
	g^\lambda_{\mu, \tilde \lambda} = c_{\mu,\lambda} \, c'_{\lambda,\mu} \, c_{\tilde \lambda} \, c'_{ \tilde \lambda}.
\]
\end{corollary}


Thanks to Theorem \ref{SThm6.3}, the previous theorem follows from

\begin{proposition}
Suppose that the diagrams of $\lambda/\mu$ and $\tilde\lambda/\tilde\mu$ coincide up to translation. Let $T$ and $\tilde T$ be standard tableaux respectively of shape $\lambda/\mu$ and $\tilde\lambda/\tilde\mu$ arising from a same filling. Then  
\[c_{\mu,\lambda} \, c'_{\lambda,\mu} \, w_{\tilde T }
=c_{\tilde\mu,\tilde\lambda} \, c'_{\tilde\lambda,\tilde\mu}\, w_{T}\]
\end{proposition}

\begin{proof}
We can assume that $\tilde \lambda/\tilde \mu$ is the minimal skew partition having the same diagram of $\lambda/\mu$ up to translation.

Let $s\in\lambda$. We analyse the contribution of $s$ in $w_T$ by distinguishing three different cases.

\begin{enumerate}

\item Suppose that both the row and the column of $s$ meet $\lambda/\mu$. Then $s$ corresponds to a box $\tilde s \in \tilde\lambda$. If $s \in \mu$, then we have $c_{\mu,\lambda,\,s}=c_{\tilde\mu,\tilde\lambda,\,\tilde s}$, $c'_{\lambda,\mu,\,s}=c'_{\tilde\lambda,\tilde\mu,\,\tilde s}$, $c_{\mu,\,s}=c_{\tilde\mu,\,\tilde s}$, $c'_{\mu,\,s}=c'_{\tilde\mu,\,\tilde s}$, and if moreover $s \in \lambda^{(i)}$ (resp. $s \in \lambda^{(i-1)}$) then we also have 
\begin{align*}
B_{\lambda^{(i)}/\lambda^{(i-1)},\,s}&= B_{\tilde\lambda^{(i)}/\tilde\lambda^{(i-1)},\,\tilde s}\\ C_{\lambda^{(i)}/\lambda^{(i-1)},\,s}& = C_{\tilde\lambda^{(i)}/\tilde\lambda^{(i-1)},\,\tilde s}\ .
\end{align*}

\item Suppose that $s \in \mu$ and that the corresponding column in $\lambda$ does not meet $\lambda/\mu$, namely $\ell_{\lambda,\,s} = \ell_{\mu,\,s}$. Then
$c_{\mu,\lambda,\, s} = c_{\mu,\,s}$ and $c'_{\lambda,\mu,\, s} = c'_{\lambda,\,s}$, and
 \begin{align*}
	 B_{\lambda^{(i)}/\lambda^{(i-1)},\,s} &= c'_{\lambda^{(i)},\,s} \\
	 C_{\lambda^{(i)}/\lambda^{(i-1)},\,s} &= c'_{\lambda^{(i-1)},\,s}
\end{align*}
for all $i$. Thus
\[
c_{\mu,\,s}\, c'_{\mu,\,s} \prod_i \frac{B_{\lambda^{(i)}/\lambda^{(i-1)},\,s}}{C_{\lambda^{(i)}/\lambda^{(i-1)},\,s}} = c_{\mu,\,s}\, c'_{\mu,\,s} \frac{c'_{\lambda,\,s}}{c'_{\mu,\,s}} = c_{\mu,\lambda,\,s}\, c'_{\lambda,\mu,\,s} \ .
\]
If $s$ corresponds to a box in $\tilde \lambda$, then the same is true for $\tilde s$:
\[
c_{\tilde\mu,\,\tilde s}\, c'_{\tilde\mu,\,\tilde s} \prod_i \frac{B_{\tilde\lambda^{(i)}/\tilde\lambda^{(i-1)},\,\tilde s}}{C_{\tilde \lambda^{(i)}/\tilde\lambda^{(i-1)},\,\tilde s}} = c_{\tilde\mu,\tilde\lambda,\,\tilde s}\, c'_{\tilde\lambda,\tilde\mu,\,\tilde s} \ .
\]

\item Suppose now that $s \in \mu$ and that the corresponding row in $\lambda$ does not meet $\lambda/\mu$, namely $a_{\lambda,\,s} = a_{\mu,\,s}$. Then
$c_{\mu,\lambda,\, s} = c_{\lambda,\,s}$ and $c'_{\lambda,\mu,\, s} = c'_{\mu,\,s}$. Denote $p = \ell_{\lambda,s} - \ell_{\mu,s}$ and assume $p > 0$. Let $i_1 < i_2 < \ldots < i_p$ be the entries of $T$ appearing in the column of $s$, and set $i_0 = 0$. Then
 \begin{align*}
	 B_{\lambda^{(i)}/\lambda^{(i-1)},\,s} &=
\left\{ \begin{array}{ll}
(a_{\lambda,\,s}+1) \alpha + \ell_{\mu,\,s} + k & \text{if } i_k < i < i_{k+1} \\
 a_{\lambda,\,s} \, \alpha + \ell_{\mu,\,s} + k +1 & \text{if } i = i_k  
\end{array} \right. \\
	 C_{\lambda^{(i)}/\lambda^{(i-1)},\,s} &= \left\{ \begin{array}{ll}
(a_{\lambda,\,s}+1) \alpha + \ell_{\mu,\,s} + k & \text{if } i_k < i < i_{k+1} \\
a_{\lambda,\,s} \, \alpha + \ell_{\mu,\,s} + k  & \text{if } i = i_k  
\end{array} \right.
\end{align*}
Thus we obtain
\[
c_{\mu,\,s} c'_{\mu,\,s} \prod_i \frac{B_{\lambda^{(i)}/\lambda^{(i-1)},\,s}}{C_{\lambda^{(i)}/\lambda^{(i-1)},\,s}} = c_{\mu,\,s} c'_{\mu,\,s} \frac{c_{\lambda,\,s}}{c_{\mu,\,s}} = c_{\mu,\lambda,\,s} c'_{\lambda,\mu,\,s}\ . 
\] 
If $s$ corresponds to a box $\tilde s \in \tilde \lambda$, then a similar equality holds for $\tilde s$.
\end{enumerate}
By the definition of $w_T$, it follows that the only boxes of $\lambda$ which give a contribution in the product $w_T/(c_{\mu,\lambda} c'_{\lambda,\mu})$ are those of the first kind. The claim follows.
\end{proof}

\subsection{Rotation of the skew diagram}
From Theorem~\ref{SThm6.3} we obtain also the following.

Fix $\beta=(b^h)$ a rectangular partition. Let $\lambda$ be a partition contained in $\beta$, that is, $\lambda_1\leq b$ and $\ell(\lambda)\leq h$. We denote here by $\hat\lambda$ the \textit{rotated complement} of $\lambda$ in $\beta$
\[\hat\lambda=(b-\lambda_h,b-\lambda_{h-1},\ldots,b-\lambda_1).\]
If $\mu\subset\lambda\subset\beta$, then the diagram of the skew partition $\hat\mu/\hat\lambda$ coincides with that of $\lambda/\mu$ up to a rotation.

\begin{theorem}	\label{rotation}
Let $\mu\subset\lambda$ be partitions and let $\hat \lambda \subset \hat \mu$ be the respective rotated complements in a rectangular partition containing $\lambda$. Then
\[c_{\mu,\lambda} \, c'_{\lambda,\mu} \, J_{\hat\mu/\hat\lambda}
=c_{\hat\lambda,\hat\mu} \, c'_{\hat\mu,\hat\lambda}\, J_{\lambda/\mu}\]
\end{theorem}

As a consequence of the previous formula, thanks to \cite{KS} we get

\begin{corollary}\label{cor:rotation}
Let $\lambda/\mu$ be a skew partition with $\lambda$ rectangular. Let $\hat\mu$ be the rotated complement of $\mu$ in $\lambda$, then
\[J_{\lambda/\mu} =c_{\mu,\lambda} \,c'_{\lambda,\mu} \,J_{\hat\mu}.\]
In particular $\tilde v_{\lambda/\mu,\,\nu}(\alpha)$ is a polynomial with nonnegative integer coefficients, for all partitions $\nu$. Furthermore, we have 
\[
	g^\lambda_{\mu, \hat \mu} = c_{\mu,\lambda} \,c'_{\lambda,\mu} \, c_{\hat \mu} \, c'_{\hat \mu}  = c_{\hat \mu,\lambda} \,c'_{\lambda, \hat \mu} \, c_\mu \, c'_ \mu.
\]
\end{corollary}

\begin{remark}
The fact that for a rectangular partition $\lambda$ the functions $J_{\lambda/\mu}$ and $J_{\hat\mu}$ are proportional was known thanks to a result by T.W.~Cai and N.~ Jing \cite[Theorem 4.7]{CJ}. They proved that if $\lambda$ is rectangular and $\mu \subset \lambda$, then $g^\lambda_{\mu, \nu} \neq 0$ if and only if $\nu$ equals $\hat \mu$, the rotated complement of $\mu$ in $\lambda$, which is equivalent to say that $J_{\lambda/\mu}$ and $J_{\hat\mu}$ differ for the rational factor $g^\lambda_{\mu, \hat \mu} / j_{\hat \mu}$. They gave also a formula expressing $g^\lambda_{\mu, \hat \mu}$ as a product of linear factors, however the polinomiality of the rational factor it is not evident from their formula.
\end{remark}

Fix $\beta=(b^h)$ a rectangular partition, let $\mu\subset\lambda\subset \beta$ and let $\hat \lambda \subset \hat \mu \subset \beta$ be the respective rotated complements in $\beta$. If $T$ is a standard tableau of shape $\lambda/\mu$ with labels in $\{1, \ldots, r\}$, we denote by $\hat T$ the standard tableau of shape $\hat\mu/\hat\lambda$ arising from the same filling, after rotating the diagram and reversing the labels $x_{\hat i}=x_{r+1-i}$ for all $i=1,\ldots,r$.

Thanks to Theorem \ref{SThm6.3}, Theorem~\ref{rotation} is an immediate consequence of the following

\begin{proposition}\label{rotation_tableau}
In the previous notation, we have 
\[c_{\mu,\lambda}\, c'_{\lambda,\mu}\,w_{\hat T}
= c_{\hat\lambda,\hat\mu} \, c'_{\hat\mu,\hat\lambda} \, w_T\]
\end{proposition}

We split the proof of Proposition~\ref{rotation_tableau} in a few lemmas. The proof will be by induction on the number $r$ of labels. The base step $r=1$, treated in the following lemma, is the case of the horizontal strips with all boxes filled with the same label.

\begin{lemma}\label{rotation_horizontal_strip}
Let $\lambda/\mu$ be a horizontal strip, then
\[\frac{c_\mu \, c'_\mu \, B_{\lambda/\mu}}{c_{\mu,\lambda} \, c'_{\lambda,\mu} \, C_{\lambda/\mu}}=
\frac{c_{\hat\lambda} \, c'_{\hat\lambda} \, B_{\hat\mu/\hat\lambda}}{c_{\hat\lambda,\hat\mu} \, c'_{\hat\mu,\hat\lambda} \, C_{\hat\mu/\hat\lambda}} \]
\end{lemma}

\begin{proof}
Let us denote by $i_1<\ldots<i_\ell$ the indices of the nonempty rows of $\lambda/\mu$, we must have for all $j=1,\ldots,\ell-1$
\[\lambda_{i_j}>\mu_{i_j}\geq\lambda_{i_{j+1}}>\mu_{i_{j+1}}\ .\]
We analyze separately the contribution of the boxes $s \in \lambda$ in the expressions appearing in the statement, respectively written as products of factors indexed by $s \in \lambda$ on the left hand side, and by $s \in \hat \mu$ on the right hand side.

\textit{Case 1.} Let $s$ be a box in $\lambda/\mu$, then $s$ gives a contribution $c_{\lambda,\,s}$ on the left hand side. Similarly, the corresponding box in $\hat \mu/\hat \lambda$ gives a contribution $c_{\hat \mu,\,s}$ on the right hand side.
Thus the product of the factors on the left hand side coming from the boxes in the $i_j$-th row of $\lambda/\mu$ gives
\[(1)(\alpha+1)(2\alpha+1)\cdots((\lambda_{i_j}-\mu_{i_j}-1)\alpha+1), \]
and the same contribution arises on the right hand side from the boxes in the corresponding row of $\hat\mu/\hat\lambda$.\\

 
\textit{Case 2.}
Suppose that $s\in\mu$ lies in a column which does not meet $\lambda/\mu$: then
\begin{align*}
c_{\mu,\lambda,\,s}&=c_{\mu,\,s}\\
B_{\lambda/\mu,\,s}&= c'_{\lambda,\,s}=c'_{\lambda,\mu,\,s}\\
C_{\lambda/\mu,\,s}&=c'_{\mu,\,s}
\end{align*}
and the contribution of $s$ to the left hand side equals 1. The same holds on the right hand side for the boxes in $\hat\lambda$ whose column does not meet $\hat\mu/\hat\lambda$.\\

\textit{Case 3.}
Suppose that $s\in\mu$ lies in a row which does not meet $\lambda/\mu$: then 
\begin{align*}
c'_{\lambda,\mu\,s}&=c'_{\mu,\,s}\\
B_{\lambda/\mu,\,s}&= c_{\lambda,\,s}=c_{\mu,\lambda,\,s}\\
C_{\lambda/\mu,\,s}&=c_{\mu,\,s}
\end{align*}
and  the contribution of $s$ to the left hand side equals 1. The same holds on the right hand side for the boxes in $\hat\lambda$ whose row does not meet $\hat\mu/\hat\lambda$.\\

\textit{Case 4.}
We are left with the boxes $s \in \mu$ (resp. $s \in \hat \lambda$) such that both the column and the row of $s$ meet $\lambda/\mu$ (resp. $s \in \hat \mu/ \hat \lambda$).

By definition, if $s \in \mu$ (resp. $s \in \hat \lambda$) is a box as above, then $B_{\lambda/\mu,s} = c_{\lambda,s}$ and $C_{\lambda/\mu,s} = c_{\mu,s}$ (resp. $B_{\hat \mu/\hat\lambda,s} = c_{\hat\mu,s}$ and $C_{\hat \mu/\hat\lambda,s} = c_{\hat\lambda,s}$). Thus the claim follows if we show that, for all $j,k$ with $1\leq j<k\leq \ell$, the next equality holds:
\[
\prod_{q=\mu_{i_{k}}+1}^{\lambda_{i_{k}}}  \frac{c'_{\mu,(i_j,q)}\, c_{\lambda,(i_j,q)}} {c_{\mu,\lambda,(i_j,q)}\, c'_{\lambda,\mu,(i_j,q)}}  \; = \;
\prod_{q=\hat\lambda_{h+1-i_{j}}+1}^{\hat\mu_{h+1-i_{j}}}\frac{c'_{\hat\lambda,(h+1-i_{k},q)}\, c_{\hat\mu,(h+1-i_{k},q)}}{c_{\hat\lambda,\hat\mu,(h+1-i_{k},q)}\, c'_{\hat\mu,\hat\lambda,(h+1-i_{k},q)}} \,.
\]
In the first term we have
\[\prod_{q=\mu_{i_{k}}+1}^{\lambda_{i_{k}}}\frac{c'_{\mu,(i_j,q)}}{c'_{\lambda,\mu,(i_j,q)}}
=\prod_{q=\mu_{i_{k}}+1}^{\lambda_{i_{k}}}\frac{(\mu_{i_j}+1-q)\alpha+i_{k}-i_{j}-1}{(\lambda_{i_j}+1-q)\alpha+i_{k}-i_{j}-1}\,,\]
\[\prod_{q=\mu_{i_{k}}+1}^{\lambda_{i_{k}}}\frac{c_{\lambda,(i_j,q)}}{c_{\mu,\lambda,(i_j,q)}}
=\prod_{q=\mu_{i_{k}}+1}^{\lambda_{i_{k}}}\frac{(\lambda_{i_j}-q)\alpha+i_{k}-i_{j}+1}{(\mu_{i_j}-q)\alpha+i_{k}-i_{j}+1} \,.\]
In the second term we have
\[\prod_{q=\hat\lambda_{h+1-i_{j}}+1}^{\hat\mu_{h+1-i_{j}}}\frac{c'_{\hat\lambda,(h+1-i_{k},q)}}{c'_{\hat\mu,\hat\lambda,(h+1-i_{k},q)}}
=\prod_{p=\mu_{i_{j}}+1}^{\lambda_{i_{j}}}\frac{(p-\lambda_{i_{k}})\alpha+i_{k}-i_{j}-1}{(p-\mu_{i_{k}})\alpha+i_{k}-i_{j}-1}\,,\]
\[\prod_{q=\hat\lambda_{h+1-i_{j}}+1}^{\hat\mu_{h+1-i_{j}}}\frac{c_{\hat\mu,(h+1-i_{k},q)}}{c_{\hat\lambda,\hat\mu,(h+1-i_{k},q)}}
=\prod_{p=\mu_{i_{j}}+1}^{\lambda_{i_{j}}}\frac{(p-\mu_{i_{k}}-1)\alpha+i_{k}-i_{j}+1}{(p-\lambda_{i_{k}}-1)\alpha+i_{k}-i_{j}+1} \,.\]

To compare the terms of the stated equality, we compare the above factors separately. By symmetry we can assume that the difference $d := (\lambda_{i_{k}}-\mu_{i_{k}})-(\lambda_{i_j}-\mu_{i_j})$ is nonnegative.

First we prove the equality
\[\prod_{q=\mu_{i_{k}}+1}^{\lambda_{i_{k}}}\frac{(\mu_{i_j}+1-q)\alpha+i_{k}-i_{j}-1}{(\lambda_{i_j}+1-q)\alpha+i_{k}-i_{j}-1}
=\prod_{p=\mu_{i_{j}}+1}^{\lambda_{i_{j}}}\frac{(p-\lambda_{i_{k}})\alpha+i_{k}-i_{j}-1}{(p-\mu_{i_{k}})\alpha+i_{k}-i_{j}-1}\,.\]
By the assumption on $d$, notice that the numerator of the right hand side divides that of the left hand side, and similarly for the denominators. Thus dividing the first hand side by the second hand side we get
\[\frac{\prod_{q=\mu_{i_{k}}+1}^{\mu_{i_{k}}+d}(\mu_{i_j}+1-q)\alpha+i_{k}-i_{j}-1}{\prod_{q=\lambda_{i_{k}}-d+1}^{\lambda_{i_{k}}}(\lambda_{i_j}+1-q)\alpha+i_{k}-i_{j}-1}
=1 \,,\]
which shows the equality.


Arguing in a similar way, it is proved also the second equality
\[ \prod_{q=\mu_{i_{k}}+1}^{\lambda_{i_{k}}}\frac{(\lambda_{i_j}-q)\alpha+i_{k}-i_{j}+1}{(\mu_{i_j}-q)\alpha+i_{k}-i_{j}+1}
=\prod_{p=\mu_{i_{j}}+1}^{\lambda_{i_{j}}}\frac{(p-\mu_{i_{k}}-1)\alpha+i_{k}-i_{j}+1}{(p-\lambda_{i_{k}}-1)\alpha+i_{k}-i_{j}+1}	 \,. \qedhere
\]
\end{proof}

The induction step in the proof of Proposition~\ref{rotation_tableau} will follow from 

\begin{lemma} \label{rotation_induction}
Let $\beta=(b^h)$ be a rectangular partition. Let $\mu\subset\lambda\subset \beta$ and let $\hat \lambda \subset \hat \mu \subset \beta$ be the respective rotated complements in $\beta$, then
\[\frac{c_{\mu,\lambda}}{c_{\mu,\beta}}=\frac{c_{\hat\lambda,\hat\mu}}{c_{\hat\lambda,\beta}} 
\quad\text{and}\quad
\frac{c'_{\lambda,\mu}}{c'_{\beta,\mu}}=\frac{c'_{\hat\mu,\hat\lambda}}{c'_{\beta,\hat\lambda}} \]
\end{lemma}

\begin{proof}
Notice that the two identities are equivalent, up to transposing the partitions. So it is enough to prove the first one.

We proceed by induction on the size of the partitions $\mu$ and $\hat\lambda$.
If $\mu=\emptyset$ and $\lambda=\beta$ the identity is obvious.

Assume the identity holds for $\mu\subset\lambda\subset\beta$. Provided $|\lambda|>|\mu|$, we have to prove that the identity holds for the partitions obtained by adding a box of $\lambda \setminus \mu$ to $\mu$, or by removing a box of $\lambda \setminus \mu$ from $\lambda$ (that is, adding a box of $\hat \mu \setminus \hat \lambda$ to $\hat \lambda$). Notice that the identity is symmetric up to taking the rotated complement in $\beta$, so it is enough to consider the first case.

Let $(i,j) \in \lambda \setminus \mu$ and let $\mu_+ \subset \lambda$ be the partition obtained by adding the box $(i,j)$ to $\mu$. We claim that 
\[\frac{c_{\mu_+,\lambda}}{c_{\mu,\lambda}}=\frac{c_{\hat\lambda,\hat\mu_+}\, c_{\mu_+,\beta}}{c_{\hat\lambda,\hat\mu}\, c_{\mu,\beta}}\,,\]
which implies the statement thanks to the induction hypothesis.

Notice that by construction we have $\mu_{i-1}\geq\mu_i+1=j$ and $\hat\lambda_k\geq b+1-j>\hat\lambda_{k+1}$ for some $k\leq h-i$ (see Figure~\ref{addition_example} for an example).

\begin{figure}\caption{}\label{addition_example}
\begin{picture}(135,130)\put(0,0){\line(1,0){135}}\put(135,0){\line(0,1){110}}
\put(0,0){$\young(\bs\bs\bs\bs\bs\bs\bs\bs\bs\bs\ \ ,\bs\bs\bs\bs\bs\bs\bs\bs\bs\ \ \ ,\bs\bs\bs\bs\bs\bs\bs\ast \ \ \  ,\bs\bs\bs\bs\bs\bs\bs\ \ ,\bs\bs\bs\bs\ \ \ \ \ ,\bs\ \ \ \ \   ,\ \ \ \ \ \   ,\ \ \ \   ,\ \ \ \   ,\        )$}
\put(-10,80){$i$}\put(82,118){$j$}
\put(-30,46){$h\!\!+\!\!1\!\!-\!\!k$}
\end{picture}
\end{figure}

On the left hand side, all the factors of the numerator and of the denominator cancel out except for those coming from the boxes in the $i$-th row of $\mu$ and $\mu_+$. Setting $\hat\lambda_0=b$, it follows that the left hand side equals
\[\frac{\prod_{\ell=0}^{k-1}\big(
\prod_{p=j-(b-\hat\lambda_{\ell+1})}
^{j-(b+1-\hat\lambda_\ell)}(p\,\alpha+h+1-i-\ell)\big)\,
\prod_{p=0}^{j-(b+1-\hat\lambda_k)}
(p\,\alpha+h+1-i-k)}
{\prod_{\ell=0}^{k-1}\big(\prod_{p=
(j-1)-(b-\hat\lambda_{\ell+1})}^{(j-1)-(b+1-\hat\lambda_\ell)}(p\,\alpha+h+1-i-\ell)\big) \,
\prod_{p=0}^{(j-1)-(b+1-\hat\lambda_k)}(p\,\alpha+h+1-i-k)}\,,\]
namely
\[\frac{\prod_{\ell=0}^{k}((\hat\lambda_\ell-b-1+j)\alpha+h+1-i-\ell)}
{\prod_{\ell=1}^{k}((\hat\lambda_{\ell}-b-1 +j)\alpha+h+2-i-\ell)}\ .\]
\\

On the right hand side, all the factors of $c_{\hat\lambda,\hat\mu_+}$ and of $c_{\hat\lambda,\hat\mu}$ cancel out except for those coming from the boxes in the $(b+1-j)$-th column of $\hat\lambda$: therefore
\[\frac{c_{\hat\lambda,\hat\mu_+}}{c_{\hat\lambda,\hat\mu}}=\frac{\prod_{\ell=1}^k((\hat\lambda_\ell-b-1+j)\alpha+h+1-i-\ell)}{\prod_{\ell=1}^k((\hat\lambda_\ell-b-1+j)\alpha+h+2-i-\ell)}\ .\]
The factors of $c_{\mu_+,\beta}$ and $c_{\mu,\beta}$ also cancel out, except for those coming from the $i$-th row of $\mu$ and $\mu^+$. Thus we get
\[\frac{c_{\mu_+,\beta}}{c_{\mu,\beta}}=\frac{\prod_{p=0}^{j-1}(p\,\alpha+h+1-i)}{\prod_{p=0}^{j-2}(p\,\alpha+h+1-i)}=(j-1)\alpha+h+1-i\,\]
and the claim follows.
\end{proof}


\begin{proof}[Proof of Proposition~\ref{rotation_tableau}] 
We proceed by induction on $r$, the number of indeterminates. The base step is Lemma~\ref{rotation_horizontal_strip}.

Suppose now that $r > 1$ and let $T$ be a standard tableau of shape $\lambda/\mu$ with $r$ labels: thus we have an increasing sequence of partitions
\[\mu=\lambda^{(0)} \subset \lambda^{(1)} \subset \ldots \subset \lambda^{(r-1)} \subset \lambda^{(r)} =\lambda\]
such that $\lambda^{(i)}/\lambda^{(i-1)}$ is a horizontal strip for all $i=1,\ldots,r$. Let us split the tableau $T$ in two standard tableaux, the tableau $S$ of shape $\lambda^{(r-1)}/\mu$ and the horizontal strip $R$ of shape $\lambda/\lambda^{(r-1)}$. Set $\lambda_-=\lambda^{(r-1)}$ and notice that 
\[\frac{w_T}{c_\mu \, c'_\mu}=\frac{w_S}{c_\mu \, c'_\mu}\, \frac{w_R}{c_{\lambda_-} \, c'_{\lambda_-}}\ .\]

Assume by the induction that the statement of the proposition holds for $S$. Then using Lemma~\ref{rotation_horizontal_strip} and writing $w_{\hat T}$ in terms of $w_{\hat R}$ and $w_{\hat S}$ we get
\begin{align*}
\frac{w_T}{c_{\mu,\lambda} \,  c'_{\lambda,\mu}}&=\frac{w_S}{c_{\mu,\lambda} \, c'_{\lambda,\mu}}\, \frac{w_R}{c_{\lambda_-}\,  c'_{\lambda_-}}\\
&=\frac{c_{\mu,\lambda_-} \, c'_{\lambda_-,\mu}}{c_{\mu,\lambda} \, c'_{\lambda,\mu}} \, \frac{w_{\hat S}}{c_{\hat\lambda_-,\hat\mu}\, c'_{\hat\mu,\hat\lambda_-}} \, \frac{c_{\lambda_-,\lambda}\, c'_{\lambda,\lambda_-}}{c_{\lambda_-} \, c'_{\lambda_-}} \, \frac{w_{\hat R}}{c_{\hat\lambda,\hat\lambda_-} \, c'_{\hat\lambda_-,\hat\lambda}}\\
&=\frac{c_{\mu,\lambda_-}\, c'_{\lambda_-,\mu}}{c_{\mu,\lambda}\, c'_{\lambda,\mu}} \, \frac{c_{\hat\lambda_-} \, c'_{\hat\lambda_-}}{c_{\hat\lambda_-,\hat\mu} \, c'_{\hat\mu,\hat\lambda_-}}\, \frac{c_{\lambda_-,\lambda} \, c'_{\lambda,\lambda_-}}{c_{\lambda_-} \, c'_{\lambda_-}} \, \frac{w_{\hat T}}{c_{\hat\lambda,\hat\lambda_-} \, c'_{\hat\lambda_-,\hat\lambda}}\ .
\end{align*}
On the other hand by Lemma~\ref{rotation_induction} we have the equalities
\[
\frac{c_{\mu,\lambda_-}}{c_{\hat\lambda_-,\hat\mu}} \,
\frac{c_{\lambda_-, \lambda}}{c_{\hat\lambda,\hat\lambda_-}}\, \frac{c_{\hat\lambda, \hat \mu}}{c_{\mu, \lambda}} \,
\frac{c_{\hat\lambda_-}}{c_{\lambda_-}} =
\frac{c'_{\lambda_-,\mu}}{c'_{\hat\mu, \hat\lambda_-}} \, 
\frac{c'_{\lambda,\lambda_-}}{c'_{\hat\lambda_-, \hat\lambda}} \,
\frac{c'_{\hat \mu, \hat\lambda}}{c'_{\lambda, \mu}} \,
\frac{c'_{\hat\lambda_-}}{c'_{\lambda_-}} = 1 \, ,
\]
and the claim follows.
\end{proof}

\subsection{The case $\ell(\mu) = 1$}

Recall the following expansion of Jack symmetric functions

\begin{proposition}{\cite[Proposition 4.2]{S}}\label{split}
Let $x=(x_1,x_2,\ldots)$ and $y=(y_1,y_2,\ldots)$ be two sets of indeterminates. If $\lambda/\mu$ is a skew partition, then
\[J_{\lambda/\mu}(x,y;\alpha)=\sum_\nu j_\nu(\alpha)^{-1}J_{\lambda/\nu}(x;\alpha)J_{\nu/\mu}(y;\alpha).\]
\end{proposition}

This has the following easy consequence for $J_{\lambda/\mu}(x;\alpha)$, under some strong assumption on $\mu$ which holds in particular when $\mu$ has only one row or is a prefix of $\lambda$ (the latter case is also contained in Theorem \ref{translation}).

\begin{proposition}	\label{prefix}
Let $\lambda/\mu$ be a skew partition and suppose that $\mu_i=\lambda_i$ for all $i<\ell(\mu)$. Then
\[v_{\lambda/\mu,\,\nu}(\alpha)=c'_{\mu}(\alpha)\, v_{\lambda,\,\mu\cup\nu}(\alpha).\]
In particular $\tilde v_{\lambda/\mu,\,\nu}(\alpha)$ is a polynomial with nonnegative integer coefficients.
\end{proposition}

\begin{proof}
Write 
\[J_\lambda(x,y;\alpha)=\sum_\nu j_\nu(\alpha)^{-1}J_{\lambda/\nu}(x;\alpha) J_\nu(y;\alpha).\]
Denote $r = \ell(\mu)$. By the assumption on $\mu$, the unique partition $\nu \subset \lambda$ such that $\nu \leq \mu$ is $\mu$ itself. Taking the coefficient of $y_1^{\mu_1} \cdots y_{r}^{\mu_r}$ in the previous expansion we get then
\[[y_1^{\mu_1} \cdots y_{r}^{\mu_r}] J_\lambda(x,y;\alpha) = \frac{v_{\mu,\mu}(\alpha)}{j_\mu(\alpha)}J_{\lambda/\mu}(x;\alpha).\]
Thus the first claim follows by Theorem \ref{norma}, and from \cite{KS} we get the second one.
\end{proof}

\section{Remarks on the factorizability of Stanley $g$-polynomials}\label{sec:factorizability}

\subsection{Transposition of the skew diagram}

Let $\omega_{-\frac{1}{\alpha}} : \Lambda \otimes \mathbb Q(\alpha) \rightarrow \Lambda \otimes \mathbb Q(\alpha)$ be the $\mathbb Q(\alpha)$-algebra automorphism defined by $\omega_{-\frac{1}{\alpha}}(p_r) = -\frac{1}{\alpha} p_r$, for $r \geq 1$.

In \cite[Theorem 3.3]{S}, a \emph{duality} formula relating $J_{\lambda'}$ and $J_\lambda$ was given. More generally, as an easy consequence we have the following formula relating $J_{\lambda'/\mu'}$ and $J_{\lambda/\mu}$ (see also \cite[VI (10.19)]{M}).

\begin{proposition}
Given a skew partition $\lambda/\mu$, we have
\[J_{\lambda'/\mu'}(x;\alpha) = (-\alpha)^{|\lambda| + |\mu|} \ \omega_{-\frac{1}{\alpha}} \Big(J_{\lambda/\mu}(x;1/\alpha)\Big)\]
Moreover, for all partitions $\nu$ it holds
$$j_{\nu'} (\alpha) = \alpha^{2|\nu|} \ j_\nu (1/\alpha) \quad\text{and}\quad g^{\lambda'}_{\mu', \nu'} (\alpha) = \alpha^{2|\lambda|} \ g^\lambda_{\mu, \nu} (1/\alpha).$$
\end{proposition}

\subsection{Factorizability of Stanley $g$-polynomials under special hypotheses}

In \cite{S} the following conjecture has been made. Let $c_{\mu, \nu}^\lambda$ denote the Littlewood-Richardson coefficient associated to the triple  $\lambda,\mu, \nu$.

\begin{conjecture}[{\cite[Conjecture 8.5]{S}}]	\label{conj:linear-factors}
Suppose that $c_{\mu, \nu}^\lambda = 1$. Then $g^\lambda_{\mu, \nu}$ is a product of linear factors. Moreover, we have
\[
	g^\lambda_{\mu, \nu} = \big(\prod_{s \in \lambda} c^*_{\lambda,s} \big) \big(\prod_{s \in \mu} c^*_{\mu,s}\big) \big(\prod_{s \in \nu} c^*_{\nu,s}\big)
\]
where for $\pi=\lambda,\mu,\nu$ and $s\in\pi$ it holds either $c^*_{\pi,s} = c_{\pi,s}$ or $c^*_{\pi,s} = c'_{\pi,s}$, and totally the two choices occur both $|\lambda|$ times.
\end{conjecture}

When $\nu$ is the highest partition occurring in $J_{\lambda/\mu}$, the conjecture is true thanks to Corollary \ref{leading-coefficient}. Indeed in that case we have
$$
	g^\lambda_{\mu, \nu} = \frac{v_{\lambda/\mu, \nu} \; j_\nu}{v_{\nu, \nu}} = v_{\lambda/\mu, \nu} \; c'_\nu.
$$
Thus Corollary \ref{leading-coefficient} yields the following description.

\begin{proposition}[{\cite[Proposition 8.6]{S}}]	\label{prop:factors-top-comp}
Suppose that $\nu$ is the shape of the maximal filling of $\lambda/\mu$. For $(i,j) \in \lambda$ denote $r_i = i - \mu'_{\lambda_i}$ and $c_j = \lambda'_j - \mu'_j$. Then 
\[
g^\lambda_{\mu, \nu} = \prod_{\tiny\begin{array}{c}(i,j)\in\lambda\\ r_i \leq c_j\end{array}} c_{\lambda,(i,j)}
\prod_{\tiny\begin{array}{c}(i,j)\in\lambda\\ r_i > c_j\end{array}} c'_{\lambda,(i,j)}
\prod_{\tiny\begin{array}{c}(i,j)\in\mu\\ r_{i+c_j} > c_j\end{array}} c_{\mu,(i,j)}
\prod_{\tiny\begin{array}{c}(i,j)\in\mu\\ r_{i+c_j}\leq c_j\end{array}} c'_{\mu,(i,j)} \quad  \prod_{s \in \nu} c'_{\nu,s}
\]
In particular, Conjecture \ref{conj:linear-factors} holds  true in this case.
\end{proposition}

In the notation of Corollary \ref{leading-coefficient}, the last statement follows by considering the map $(i,j) \mapsto (i+c_j, j)$, which defines a bijection between the boxes of $\mu$ and the boxes of $\lambda/\nu$.

By duality, we get a similar formula also for $g^{\lambda'}_{\mu',\nu'}(\alpha) = \alpha^{2|\lambda|} \; g^\lambda_{\mu,\nu}(\frac{1}{\alpha})$. Notice that $\nu$ is the weight of the highest filling of $\lambda/\mu$ if and only if $\lambda' = \mu' + \sigma(\nu')$ for some permutation $\sigma$ of the rows of $\nu'$.

\begin{proposition}	\label{prop:prv}
Suppose that  $\lambda = \mu + \sigma(\nu)$, where $\sigma$ is a permutation of the rows of $\nu$. For $(i,j) \in \lambda$ denote $r_i = \nu_{\sigma^{-1}(i)}$ and $c_j = j - \mu_{\lambda'_j}$. Then 
\[
g^\lambda_{\mu, \nu} = \prod_{\tiny\begin{array}{c} (i,j)\in\lambda\\ c_j \leq r_i \end{array}} c'_{\lambda,(i,j)}
\prod_{\tiny\begin{array}{c}(i,j)\in\lambda \\ c_j > r_i \end{array}} c_{\lambda,(i,j)}
\prod_{\tiny\begin{array}{c}(i,j)\in\mu \\ c_{j+r_i} > r_i \end{array}} c'_{\mu,(i,j)}
\prod_{\tiny\begin{array}{c}(i,j)\in\mu\\ c_{j+r_i}\leq r_i\end{array}} c_{\mu,(i,j)} \quad 
\prod_{s \in \nu} c_{\nu,s}
\]
In particular, Conjecture \ref{conj:linear-factors} holds true in this case.
\end{proposition}

In the notation of the previous proposition, notice that if $(i,j) \in \lambda$ and $c_j > r_i$ then $i < \lambda'_j$ is an inversion of $\sigma^{-1}$.

The previous formula for $g^\lambda_{\mu,\nu}$ can also be deduced by a more general result of Ruitenburg \cite[Appendix]{R}, which holds for Jacobi polynomials for arbitrary root systems.

\subsection{Skew diagram consisting of two connected components of nonskew type}

\begin{conjecture}	\label{conj:linear-factors-ratio}
Let $\mu$ be a rectangular partition and let $\lambda\supset\mu$ be a partition such that $\lambda/\mu$ has two connected components given by the partitions $\phi$ and $\psi$. Let $\nu$ be such that $c^\lambda_{\mu,\nu} = 1$. Then the ratio $g^{\lambda}_{\mu,\nu}/g^{\nu}_{\phi,\psi}$ decomposes into linear factors compatibly with Conjecture \ref{conj:linear-factors}:
\[
	\frac{g^\lambda_{\mu, \nu}}{g^{\nu}_{\phi,\psi}} = \prod_{s \in \mu} c^*_{\lambda,s} c^*_{\mu,s}
\]
where for $\pi=\lambda,\mu$ and $s \in \mu$ it holds either $c^*_{\pi,s} = c_{\pi,s}$ or $c^*_{\pi,s} = c'_{\pi,s}$, and totally the two choices occur both $|\mu|$ times.
\end{conjecture}

\begin{proposition}
Let $\mu$ be a rectangular partition and let $\lambda \supset \mu$ be a partition such that $\lambda/\mu$ has two connected components given by the partitions $\phi$ and $\psi$. Then Conjecture \ref{conj:linear-factors-ratio} holds true
for the partitions $\nu = \phi + \psi$ and $\phi \cup \psi$.
\end{proposition}

\begin{proof}
Assume that $\nu = \phi + \psi$, the other case is treated similarly (or using the duality). Assume that $\ell(\phi) \leq \ell(\psi)$, and apply Proposition \ref{prop:prv} to the triple $\{\nu,\phi, \psi\}$, with $\sigma$ the trivial permutation. 
%
%
%
Then we get
\[
	g^{\nu}_{\phi,\psi} = \big(\prod_{s \in \nu} c'_{\nu,s} \big) \; 
	\big(\prod_{s \in \phi} c_{\phi,s} \big) \;
	 \big(\prod_{s \in \psi} c_{\psi,s} \big).
\]
On the other hand
\[
	\big(\prod_{s \in \phi} c_{\nu_1,s} \big) \; 	
	\big(\prod_{s \in \psi} c_{\nu_2,s} \big) = \prod_{s \in \lambda/\mu} c_{\lambda,s},
\]
and with respect to the filling on $\lambda/\mu$ defined by $\nu$ we have $r_i \leq c_j$ for all $(i,j) \in \lambda/\mu$. Thus by Proposition \ref{prop:factors-top-comp} we see that
$g^\lambda_{\mu, \nu}$ is divisible by $g^{\nu}_{\phi,\psi}$, and their ratio is a product of $2|\mu|$ factors of the desired shape.
\end{proof}

\begin{remark}
Let us here comment on Conjecture \ref{conj:linear-factors-ratio} in parallel with Theorem \ref{translation} and Theorem \ref{rotation}. 

First recall the invariance of skew Schur symmetric functions by translation and by rotation of the skew diagram
\[s_{\tilde\lambda/\tilde\mu}=s_{\lambda/\mu}\quad\text{and}\quad s_{\hat\lambda/\hat\mu}=s_{\lambda/\mu}.\]
Theorem \ref{translation} and Theorem \ref{rotation} generalize these identities by stating the semi-invariance of skew Jack symmetric functions by translation and by rotation.

Recall now the equality of a skew Schur symmetric function $s_{\lambda/\mu}$ with the product of the skew Schur symmetric functions associated with the connected components of the skew diagram of $\lambda/\mu$. We can also restrict to the special case, as in Conjecture \ref{conj:linear-factors-ratio}, of a rectangular partition $\mu$ such that $\lambda/\mu$ has two connected components given by the partitions $\phi$ and $\psi$, we have
\[s_{\lambda/\mu}=s_{\phi}\,s_{\psi},\]
while in general the skew Jack symmetric function $J_{\lambda/\mu}$ is not proportional to the product of the Jack symmetric functions $J_\phi\,J_\psi$. Conjecture \ref{conj:linear-factors-ratio} is a partial attempt to understand the relation between $J_{\lambda/\mu}$ and $J_\phi\,J_\psi$.
\end{remark}

\section{Looking for a generalization of Knop and Sahi's combinatorial formula}\label{sec:lowest}

Let us recall the integral combinatorial formula due to Knop and Sahi \cite{KS} for the Jack symmetric functions $J_\lambda(x;\alpha)$.

A (not necessarily standard) tableau $T$ of shape $\lambda$ is called {\it admissible} if for all boxes $(i,j)\in\lambda$:
\begin{itemize}
\item $T(i,j)\neq T(i',j)$ for all $i'>i$,
\item $T(i,j)\neq T(i',j-1)$ for all $i'<i$ and $j>1$.
\end{itemize} 

A box $(i,j)\in\lambda$ is called {\it critical} for $T$ if $j>1$ and $T(i,j)=T(i,j-1)$.

For $s\in\lambda$, set 
\[d_{\lambda,\,s}=d_{\lambda,\,s}(\alpha)=(a_{\lambda,\,s}+1)\alpha+\ell_{\lambda,\,s}+1\]
and, for $T$ tableau of shape $\lambda$, set
\[d_T=d_T(\alpha)=\prod_{s\text{ critical}}d_{\lambda,\,s}(\alpha)\ .\]

\begin{theorem}{\cite[Theorem 5.1]{KS}}\label{KSformula}
\[J_\lambda(x;\alpha)=\sum_{T \; \mathrm{admissible}}d_T(\alpha)x^T\ .\]
\end{theorem}

We are not able to formulate a conjecture for a generalization of Theorem \ref{KSformula} to skew Jack symmetric functions. Here we just formulate a rather intricate combinatorial conjecture only for the lowest coefficient.

\subsection{The lowest coefficient}

Let us look at the lowest coefficient of $J_{\lambda/\mu}(x)$ with respect to the monomial symmetric functions, that is, the function $v_{\lambda/\mu,\,(1^n)}(\alpha)$ where $(1^n)$ denotes the one column partition $(1,\ldots,1)$ of length $n=|\lambda|-|\mu|$. From the definition we have
\[\frac{v_{\lambda/\mu,\,(1^n)}(\alpha)}{n!}=\sum_\nu\frac{\langle J_\lambda(x),J_\mu(x) J_\nu(x)\rangle}{j_\nu}.\] 

A (finite) subset $C$ of $\mathbb Z_{>0}\times\mathbb Z_{>0}$ will here be called a {\it configuration}. Generalizing Young diagrams, we can think of a configuration as a set of boxes at integer positions in the positive quadrant. 

For any configuration $C$ we define two partitions: $\rho(C)$, the numbers $r_i(C)$ ($i=1,2,\ldots$) of boxes of $C$ in the $i$-th row rearranged in decreasing order, and $\gamma(C)$, the numbers $c_j(C)$ ($j=1,2,\ldots$) of boxes of $C$ in the $j$-th column rearranged in decreasing order. A configuration (of $|\mu|$ boxes) will be called $\mu$-{\it admissible} if it can be obtained from the Young diagram of $\mu$ by a (possibly empty) sequence of moves of the following kinds:
\begin{enumerate}
\item moving a box of the configuration $C$ along the same row obtaining a new configuration $C'$ with $\gamma(C')<\gamma(C)$,
\item moving a box of the configuration $C$ along the same column obtaining a new configuration $C'$ with $\rho(C')<\rho(C)$,
\item permuting the rows of the configuration,
\item permuting the columns of the configuration.
\end{enumerate}
   
By explicit computations on small partitions we have observed the following

\begin{conjecture}For all partitions $\mu$ and all $\mu$-admissible configurations $C$ there exist (uniquely determined) polynomials $\pi_{(\mu,C)}(\alpha)$ with nonnegative integer coefficients, invariant by row permutations and column permutations of $C$, such that for all partitions $\lambda$
\[\frac{v_{\lambda/\mu,\,(1^n)}(\alpha)}{n!}=\sum_{C\subset\lambda} \pi_{(\mu,C)}(\alpha)\]
with $C$ varing among $\mu$-admissible configurations included in the Young diagram of $\lambda$. The polynomials do not depend on $\lambda$ but only on $\mu$ and $C$.
\end{conjecture}


\begin{thebibliography}{99}
%
\bibitem{BG} P.\ Bravi and J.\ Gandini, \textit{On the multiplication of spherical functions of reductive spherical pairs of type A},  	arXiv:2106.04893~.
%
\bibitem{CJ} T.W.\ Cai and N.\ Jing, \textit{Jack vertex operators and realization of Jack functions}, J.\ Algebr.\ Comb. \textbf{39} (2014), 53--74.
%
%
\bibitem{GH} W.~Graham and M.~Hunziker, \textit{Multiplication of polynomials on Hermitian symmetric spaces and Littlewood-Richardson coefficients}, Canad.\ J.\ Math. \textbf{61} (2009), 351–372.
%
\bibitem{KS} F.\ Knop and S.\ Sahi, \textit{A recursion and a combinatorial formula for Jack polynomials}, Invent.\ Math. \textbf{128} (1997), 9--22.
%
\bibitem{M} I.G.\ Macdonald, \textit{Symmetric Functions and Hall Polynomials}, 2nd edition, Oxford Science Publications (1995).
%
%
\bibitem{R} G.C.M.~Ruitenburg, \textit{Invariant ideals of polynomial algebras with multiplicity free group action}, Compositio Math. \textbf{71} (1989), 181--227.
%
\bibitem{SAGE} Sage Mathematical Software System, \verb|https://www.sagemath.org|~.
%
\bibitem{S} R.\ Stanley, \textit{Some combinatorial properties of Jack symmetric functions}, Adv.\ Math. \textbf{77} (1989), 76--115.
%
\end{thebibliography}
\end{document}